\documentclass{amsart}

\usepackage{amssymb,amscd,amsthm,latexsym, amsmath}
\usepackage[all]{xy}

\theoremstyle{plain}
\newtheorem{Lemma}{Lemma}[section]
\newtheorem{theorem}[Lemma]{Theorem}

\newtheorem{proposition}[Lemma]{Proposition}
\newtheorem{corollary}[Lemma]{Corollary}
\theoremstyle{definition}
\newtheorem{definition}[Lemma]{Definition}
\newtheorem{example}[Lemma]{Example}

\theoremstyle{remark}
\newtheorem{remark}[Lemma]{Remark}

\DeclareMathOperator{\End}{End}

\newdir{ >}{{}*!/-8pt/@{>}}

\def\Z{ \mathbb{Z} }

\def\cleft{\hbox{[\kern-.16em\hbox{[}}}
\def\cright{\hbox{]\kern-.16em\hbox{]}}}

\newcommand{\Ker}{ \ensuremath{\mathrm{Ker}} }

\newcommand{\Gp}{\mathsf{Gp}}
\newcommand{\Set}{\mathsf{Set}}

\newcommand{\DiGp}{\mathsf{DiGp}}
\newcommand{\rann}{\operatorname{r.ann}}
\newcommand{\lann}{\operatorname{l.ann}}
\DeclareMathOperator{\Aut}{Aut}

\renewcommand{\phi}{\varphi}

\DeclareMathOperator{\sgn}{sgn}
\DeclareMathOperator{\id}{id}
\numberwithin{equation}{section}

\begin{document}
	
\title{Semidirect products of skew braces}


\author{Alberto Facchini}
\address{Dipartimento di Matematica ``Tullio Levi-Civita'', Universit\`a di Padova, 35121 Padova, Italy}
\email{facchini@math.unipd.it}
\thanks{The first author was partially supported by Ministero dell'Universi\-t\`a e della Ricerca (Progetto di ricerca di rilevante interesse nazionale ``Categories, Algebras: Ring-Theoretical and Homological Approaches (CARTHA)''), Fondazione Cariverona (Research project ``Reducing complexity in algebra, logic, combinatorics - REDCOM'' within the framework of the programme Ricerca Scientifica di Eccellenza 2018), and the Department of Mathematics ``Tullio Levi-Civita'' of the University of Padua (Research programme DOR1828909 ``Anelli e categorie di moduli''). The second author was partially supported by Austrian Science Fund (FWF), Project Number W1230.}

\author{Mara Pompili}
\address{University of Graz, Institute for Mathematics and Scientific Computing, Heinrichstrasse 36, 8010 Graz, Austria}
\email{mara.pompili@uni-graz.at}


   \begin{abstract} We study the notions of action, semidirect product and commutator of ideals for digroups and left skew braces.\end{abstract}

	\maketitle
	
	\section{Introduction}
	
	This paper is devoted to the study of semidirect products of left skew braces and, more generally,
digroups. Semidirect products are in strict relation with  idempotent endomorphisms and actions. We also describe the notions of commutator of ideals and center for digroups and skew braces.

Recall that a {\em digroup} is a triplet $(D,*,\circ)$, where $(D,*) $ and $(D,\circ)$ are groups   with the same identity $1_D$. A {\sl (left) skew brace}  \cite{GV}  is
		a digroup $(D, *,\circ)$ for which \begin{equation}a\circ (b * c) = (a\circ b)*a^{-*}* ( a\circ c)\label{lsb}\tag{B}\end{equation}
		for every $a,b,c\in D$. Here $a^{-*}$ denotes the inverse of $a$ in the group $(D,*)$. 
		
		Left skew braces were introduced in 2017 by Guarnieri and Vendramin \cite{GV} in the study of the set-theoretic solutions
of the Yang-Baxter equation, that is, of the pairs $(X, r)$, where $X$ is a set,
$r \colon X\times X \to X \times X$ is a bijection, and 
$$(r\times\id)(\id\times r)(r\times\id)=(\id\times r)(r\times\id)(\id\times r).$$
The study of the set-theoretic solutions of  the Yang-Baxter equation was proposed by 
Drinfeld  in 1992 \cite{Drinfeld}. Left skew braces allow to determine non-degenerate solutions of the Yang-Baxter equation, i.e., the solutions $(X, r)$ for which the mappings $\pi_1\circ r(x,-)\colon X\to X$ and $\pi_2\circ r(-,y)\colon X\to X$ are all bijections, for every $x,y\in X$. Here $\pi_1$ and $\pi_2$ are the two canonical projections $X\times X\to X$.
		
				Given a digroup $(D,*,\circ)$ and an element $a$ of $D$, let $\lambda^D_a\colon D\to D$ denote the mapping defined by
$$\lambda^D_a(b)=a^{-*}*(a\circ b)$$ for every $b\in D$. Then every $\lambda^D_a$ is a bijection. The mapping $$\lambda^D\colon D\to S_D,\qquad \lambda^D\colon a\mapsto\lambda^D_a,$$ is a mapping of the set $D$ into the symmetric group $S_D$ of all permutations of  the set $D$. More precisely, $\lambda^D$ is a morphism $\lambda^D\colon (D,1)\to (\Aut_{\Set_*}(D,1),\mathrm{Id})$ in the category $\Set_*$ of all pointed sets, of $(D,1)$ into the automorphism group $ \Aut_{\Set_*}(D,1)$ of all automorphisms of  the pointed set $(D,1)$. 
The digroup $(D,*,\circ)$ is a left skew brace if and only if $\lambda^D$ is a group morphism of $(D,\circ)$ into the group $\Aut_\Gp(D,*)$ or, equivalently, when $\lambda^D_a$ is a group automorphism $(D,*) \to (D,*)$ for all $a\in D$.  

		In this paper we deepen some aspects of the study of digroups and left skew braces we began to investigate with Dominique Bourn in \cite{BFP}. Left skew braces show an unexpected algebraic richness. 

One of the aspects of digroups and skew braces we hadn't considered in \cite{BFP} and we do consider here is that of semidirect product of digroups and left skew braces, a
notion that has already appeared in the literature of skew braces \cite[Definition 2.2]{Kinnear}, \cite{R46}, \cite[Corollaries 2.36 and 2.37]{SV}, and \cite{Rathee}. Here we give a systematic treatment of semidirect product of left skew braces, also considering the case of digroups. 

For groups it is well known that idempotent endomorphisms $e$ of a group $X$ are in one-to-one correspondence with the semidirect-product decompositions $X=Y\ltimes_{\phi} K$ where $K=\Ker(e)$, $Y= e(X)$ and $\phi\colon Y\to\Aut(K)$ is the action of $Y$ on $K$ defined by $\phi_y(k)=y^{-1} k y$. This $\phi$ is a group antihomomorphism. The operation on $Y\ltimes_{\phi} K$ is defined by $(y,k)(y',k')=(y y', \phi_{y'}(k) k')$.
We will see in this paper that, in the categories of digroups (and left skew braces, resp.), it is natural to define an {\em action} of a digroup $Y$ on a digroup $K$ (of a left skew brace $Y$ on a left skew brace $K$, resp.) as a triplet $(\phi_*,\phi_{\circ},\Lambda)$ where $\phi_*\colon(Y,*)\to \Aut(K,*)$ and $\phi_{\circ}\colon( Y,\circ)\to \Aut_\Gp(K,\circ)$ are two group antihomomorphisms and $\Lambda\colon (Y,1)\to (\Aut_{\Set_*}(K,1),\mathrm{Id})$ is a morphism, in the category $\Set_*$ of all pointed sets, of  $(Y,1)$ into the automorphism group $ \Aut_{\Set_*}(K,1)$ (where $\Lambda\colon( Y,\circ)\to \Aut_\Gp(K,*)$ is a group homomorphism, respectively. For left skew braces, we will see in the next paragraph that a suitable identity must be also satisfied.)

Given such a digroup action $(\phi_*,\phi_{\circ},\Lambda)$, define  two operations $+$ and $\ltimes_\circ$ on the cartesian product $Y\times K$ via $$(y,k)+(y',k')=(y* y', (\Lambda_{y* y'})^{-1}(\phi_{* y'}(\Lambda_y(k))* \Lambda_{y'}(k')))$$ and $$(y,k)\ltimes_\circ(y',k')=(y\circ y', \phi_{\circ y'}(k)\circ k')$$ for every $(y,k),(y',k')\in Y\times K$. This is the semidirect product $(Y\ltimes K,+,\ltimes_\circ)$ of the digroups $K$ and $Y$, and we show in Proposition~\ref{hl} that if $Y$ and $K$ are a pair of left skew braces and $(\phi_*,\phi_{\circ},\Lambda)$ is an action of digroups of  $Y$ on $K$, then the digroup $(Y\ltimes K,+,\ltimes_\circ)$ is a left skew brace if and only if $\Lambda$ is a group homomorphism $\Lambda\colon( Y,\circ)\to \Aut_\Gp(K,*)$ and
	$$\begin{array}{l}
        \phi_{\circ (y'*y'')}(k)\circ (\Lambda_{y'*y''})^{-1}\bigl(\phi_{*y''}(\Lambda_{y'}(k'))*\Lambda_{y''}(k'')\bigr) = \\ \ \ =
        (\Lambda_{y\circ (y'*y'')})^{-1}\bigl(\phi_{*\lambda_{y}^Y(y'')}(\Lambda_{y\circ y'}(\phi_{\circ y'}(k)\circ k')*\Lambda_{y}(k)^{-*})\bigr)*\Lambda_{y\circ y''}(\phi_{\circ y''}(k)\circ k'')
    \end{array}$$ for every $y,y',y''\in Y$, $k,k',k''\in K$.

We are grateful to Professor D. Bourn for several discussions about the content of this paper.

	\section{Elementary facts about idempotent endomorphisms, and splitting  idempotents}\label{2}
	
	The equivalence relations on a digroup $(D, *,\circ)$ that are compatible with both group operations (that is, congruences of the algebra $(D, *,\circ)$) are in one-to-one correspondence with {\em ideals} $I$ of $D$, that is, normal subgroups $I$ of $(D,*)$ that are also normal subgroups of $(D, \circ)$ and have the property that $a*I=a\circ I$ for every $a\in D$ (the cosets modulo $I$ with respect to the two operations coincide) \cite{BFP}. 
		There is a forgetful functor $U$ of the category $\DiGp$ of digroups into the category $\Set_*$ of pointed sets, which assigns to each digroup $D$ the pointed set $(D$,1).
	
	Given a digroup $(D,*,\circ)$ and an element $a$ of $D$, let $\lambda^D_a\colon D\to D$ denote the mapping defined by
$$\lambda^D_a(b)=a^{-*}*(a\circ b)$$ for every $b\in D$. Here $a^{-*}$ denotes the inverse of $a$ in the group $(D,*)$, and $a^{-\circ}$ denotes the inverse of $a$ in $(D,\circ)$. Then every $\lambda^D_a$ is a bijection; more precisely, it is an automorphism of $U(D)$ in the category $\Set_*$. Its inverse is the morphism $(D,1)\to (D,1)$, $c\mapsto a^{-\circ}\circ(a*c)$. (This is the mapping $\lambda^D_a$ relative to the digroup $(D,\circ,*)$ with the operation $*$ and $\circ$ swapped. Notice that the digroups  $(D,*,\circ)$ and $(D,\circ,*)$ are not isomorphic in the category $\DiGp$.) The mapping $\lambda^D\colon (D,1)\to (\Aut_{\Set_*}(D,1),\mathrm{Id})$ is a morphism in the category $\Set_*$ of the pointed set $(D,1)$ into the automorphism group $ \Aut_{\Set_*}(D,1)$ of all automorphisms of  the pointed set $(D,1)$, that is, the
 group of all the permutations of  the set $D$ that fix~$1$. The property ``$a*I=a\circ I$ for every $a\in D$'' in the definition of ideal $I$ of a digroup $D$ is equivalent to ``$\lambda_a^D(I)\subseteq I$  for every $a\in D$''.
The digroup $(D,*,\circ)$ is a (left) skew brace if and only if the image $\lambda^D(D)$ of the mapping $\lambda^D$ is contained in the monoid of all endomorphisms of the group $(D,*)$, equivalently if and only if
$\lambda^D(D)\subseteq\Aut(D,*)$, that is, if and only if $\lambda^D$ is a group morphism of $(D,\circ)$ into the group $\Aut(D,*)$.  When it is clear from the context, we will write $\lambda$ instead of $\lambda^D$. 

\begin{remark}{\rm A similar situation occurs for rings. Assume that we have an abelian additive group $(R,+)$ with a further binary operation $\cdot\colon R\times R\to R$, $\cdot\colon (a,b)\mapsto ab$. For an element $a$ of $R$, let $\lambda_a\colon R\to R$ denote the mapping defined by
$\lambda_a(b)=ab$ for every $b\in R$. Then $\lambda_a$ is a mapping $R\to R$, that is, a morphism in the category $\Set$. The mapping $\lambda\colon R\to R^R$, $\lambda\colon a\mapsto\lambda_a$, is a morphism in $\Set$ of the set $R$ into the monoid $R^R$ of all mappings $R\to R$.
Then right and left distributivities in $(R,+,\cdot)$ hold if and only if  the image $\lambda(R)$ of the mapping $\lambda$ is contained in the ring $\End(R,+)$ of all the endomorphisms of the group $(R,+)$ and $\lambda$ is a group morphism of $(R,+)$ into the additive group $\End(R,+)$. If these equivalent conditions holds, then $\cdot$ is associative if and only if the group morphism
$\lambda$ also respects the multiplications of $R$ and $\End(R,+)$, i.e., in this case $\lambda$ is a ring morphism of $R$ into the endomorphism ring $\End(R,+)$.}\end{remark}

	Fix a digroup $(D,*,\circ)$ and an idempotent endomorphism $e$ of $D$. Let $B$ be the image of $e$, so that $B$ is a  subdigroup of $D$, and let $I$ be the kernel of $e$, which is an ideal of $D$. Since $e$ is an idempotent group endomorphism with respect to both group structures on $D$, we get that $D$ is a semidirect product of $B$ and $I$ with respect to both groups structures on $D$, so that $D=I*B=\{\, i*b\mid i\in I,b\in B\,\}$, $D=I\circ B=\{\, i\circ b\mid i\in I,b\in B\,\}$, $I\cap B=\{1_D\}$.

\medskip

Conversely, assume that the digroup $D$ has a subdigroup $B$ and an ideal $I$ such that $D$ is the semidirect product with respect to the group structure $(D,\circ)$ (i.e., $B\circ I=D$ and $B\cap I=\{1_D\}$). Then every element $a\in D$ can be written in a unique way as $a=b\circ i$. Clearly, the mapping $e\colon a=b\circ i\mapsto b$ is an idempotent group endomorphism of $(D,\circ)$. Let us prove that this mapping $e$ is a digroup endomorphism. We must show that $e(a_1*a_2)=e(a_1)*e(a_2)$ for every $a_1,a_2\in D$, that is, that for every $b_1,b_2\in B$ and every $i_1,i_2\in I$ one has that $e((b_1\circ i_1)*(b_2\circ i_2))=b_1*b_2$. Equivalently, it suffices to show that, for every $b_1,b_2\in B$, one has $(b_1\circ I)*(b_2\circ I)\subseteq (b_1*b_2)\circ I$. But $(b_1\circ I)*(b_2\circ I)=(b_1*I)*(b_2*I)=(b_1*b_2)*I= (b_1*b_2)\circ I$, as desired.

This proves the equivalence of Conditions (1), (2), (3) and  (8) in the following Proposition.

\begin{proposition} \label{0.1} Let $(D,*,\circ)$ be  a digroup (resp.~left skew brace), $B$ a subdigroup (resp.~subbrace) of $D$ and $I$ an ideal of $D$. The following conditions are equivalent:

	\begin{enumerate}
		\item $D=B\circ I$ and $B\cap I = \{1_D\}$.
		\item For every $a\in D$, there are a unique element $b\in B$ and a unique element $i_1\in I $ such that $a=b\circ i_1$.
		\item For every $a\in D$, there are a unique element $b\in B$ and a unique element $i_2\in I $ such that $a=i_2\circ b$.
		\item $D=B*I$ and $B\cap I = \{1_D\}$.

		\item For every $a\in D$, there are a unique element $b\in B$ and a unique element $i_3\in I $ such that $a=b*i_3$.
		\item For every $a\in D$, there are a unique element $b\in B$ and a unique element $i_4\in I $ such that $a=i_4* b$.
		\item There exists a digroup (resp.~left skew brace) morphism $D\to B$ whose restriction to $B$ is the identity and whose kernel is $I$.  
		\item There is an idempotent digroup (resp.~left skew brace) endomorphism of $D$ whose image is $B$ and whose kernel is $I$.
	\end{enumerate}
 \end{proposition}
 
 \begin{proof} (4)${}\Rightarrow{}$(8)  If (4) holds, the group $(D,*)$ is the semidirect product of its subgroups $B$ and $I$, so that for every element $a\in D$, there are unique $b\in B$ and $i_3\in I $ such that $a=b*i_3$, and there is an idempotent endomorphism $e$ of the group $(D,*)$ such that $e\colon a=b*i_3\mapsto b$. Let us show that $e$ is a group endomorphism of $(D,\circ)$ as well. We must show that $e(a_1\circ a_2)=e(a_1)\circ e(a_2)$. Equivalently, we must prove that $e((b_1*i_1)\circ (b_2*i_2))=b_1\circ b_2$. This holds because $(b_1*I)\circ(b_2*I)=(b_1\circ I)\circ(b_2\circ I)=(b_1\circ b_2)\circ I=(b_1\circ b_2)*I$.
 
 It is now clear that (1)--(6) and (8) are equivalent.

 (7)${}\Rightarrow{}$(8) If there is a digroup morphism $\varphi \colon D\to B$ whose restriction to $B$ is the identity and whose kernel is $I$, and $\varepsilon\colon B\to D$ is the embedding, then $\varepsilon\varphi$ is the required idempotent digroup endomorphism of the digroup $D$.

(8)${}\Rightarrow{}$(7)  If $e$ is an idempotent endomorphism satisfying (8), then the image of $e$ is $B$, and its corestriction $e|^{B}\colon D\to B$ satisfies (7).\end{proof}

We will say that the digroup $D$ is the {\em inner semidirect product} of its subdigroup $B$ and its ideal $I$ if any of the equivalent conditions of Proposition~\ref{0.1} holds.
As a corollary of Proposition~\ref{0.1}, we get that there is a bijection between the set of all idempotent digroup endomorphisms of a digroup $D$ and the set of all pairs $(B,I)$, where $B$ is a  subdigroup of $D$, $I$ is an ideal of $D$ and $D$ is the semidirect product of $B$ and $I$. This bijection associates with each idempotent endomorphism $e$ the pair $(e(D),\ker(e))$.
	
	Notice that we have shown that if $(D,*,\circ)$ is a digroup, $e$ is an idempotent group endomorphism of either $(D,*)$ or $(D,\circ)$, the kernel of $e$ is an ideal of the digroup $(D,*,\circ)$ and the image of $e$ is a subdigroup, then $e$ is an idempotent digroup endomorphism of $(D,*,\circ)$. 
	
	

	\begin{remark}\label{2.3}{\rm In Proposition~\ref{0.1}, we saw that every element $a\in D$ can be written in a unique way in the four forms $a=b\circ i_1=i_2\circ b=b* i_3=i_4* b$, where $b=e(a)$ and $i_1,i_2,i_3,i_4\in I$. Now: 
	
	\begin{enumerate}
		\item[(i)] $b\circ i_1=b* i_3$ implies that $i_3=b^{-*}*(b\circ i_1)=\lambda_b(i_1)$. 
		\item[(ii)]$b\circ i_1=i_2\circ b$ implies that $i_1=b^{-\circ}\circ i_2\circ b$, so that   $i_1=\varphi_{\circ b}(i_2)$, where $\varphi_{\circ}\colon (D,\circ)\to \Aut_\Gp(D,\circ)$ denotes the conjugation in the group $(D,\circ)$. Therefore $i_2=(\varphi_{\circ b})^{-1}(i_1)$.
		\item[(iii)]Similarly $b* i_3=i_4* b$ implies that $i_4=(\varphi_{*b})^{-1}(i_3)$, where $\varphi_{*}\colon (D,*)\to \Aut_\Gp(D,*)$ denotes the conjugation in the group $(D,*)$.  From (i), it follows that $i_4=(\varphi_{* b})^{-1}(\lambda_b(i_1))$. 
	\end{enumerate}
	
	We have thus proved that, in the statement of Proposition~\ref{0.1}, the elements $i_2,i_3,i_4$ depend only on $i_1$, via the formulas \begin{equation}\left\{\begin{array}{l} i_2=(\varphi_{\circ b})^{-1}(i_1) \\ i_3=\lambda_b(i_1) \\ i_4=(\varphi_{* b})^{-1}(\lambda_b(i_1)).\end{array}\right.\label{sistema}\end{equation} Of course, it is possible to show that any of the three elements $i_1,i_2,i_3,i_4$ determines the other three.}
	 \end{remark}
	
\section{Inner semidirect products  of digroups and left skew braces}\label{7}

Suppose we have a digroup $(D,*,\circ)$. We want to describe the semidirect-product decompositions of $D$. In view of what we have seen in Section~\ref{2}, we will study the idempotent digroup endomorphisms $e$ of $D$. Such an idempotent endomorphism $e$ will have a kernel $K$ and an image $Y$, and $D$ will turn out to be a semidirect product $D=Y\ltimes K$. In view of Remark~\ref{2.3},  three mappings immediately appear in a natural way:

(1) The group antihomomorphism $\phi_*\colon (Y, *)\to\Aut_\Gp(K,*)$ related to the semi\-direct-product decomposition $D=Y\ltimes K$ of the group $(D,*)$. It is defined by $\phi_*\colon y\mapsto\phi_{*y}$, where $\phi_{*y}(k)=y^{-*}*k*y$ for every $y\in Y$ and every $k\in K$. It is induced by the conjugation in the group $(D,*)$.

(2) The group antihomomorphism $\phi_\circ\colon (Y, \circ)\to\Aut_\Gp(K,\circ)$ related to the semi\-direct-product decomposition $D=Y\ltimes K$ of the group $(D,\circ)$. It is defined by $\phi_\circ\colon y\mapsto\phi_{\circ y}$, where $\phi_{\circ y}(k)=y^{-\circ}\circ k\circ y$ for every $y\in Y$, $k\in K$. It is induced by the conjugation in the group $(D,\circ)$.

(3) The morphism $\Lambda\colon (Y,1)\to (\Aut_{\Set_*}(K,1),\mathrm{Id})$ in the category $\Set_*$, defined by $\Lambda\colon y\mapsto\Lambda_y$, where $\Lambda_y(k)=y^{-*}*(y\circ k)$ for every $y\in Y$, $k\in K$. It is induced by the morphism $\lambda\colon  (D,1)\to (\Aut_{\Set_*}(D,1),\mathrm{Id})$ ($\lambda$ is a morphism in $\Set_*$), relative to the digroup $(D,*,\circ)$ (Section~\ref{2}). Here $\Aut_{\Set_*}(X,x_0)$ denotes the group of all automorphisms of a pointed set $(X,x_0)$. Each $\Lambda_y$ is a permutation of $K$ because its two-sided inverse is the mapping $K\to K$, $k\mapsto y^{-\circ}\circ(y*k)$. Notice that $\Lambda_1\colon K\to K$ is the identity mapping of $K$ and $\Lambda_y(1)=1$ for every $y\in Y$, i.e., that $\Lambda$ is a morphism $\Lambda\colon (Y,1)\to (\Aut_{\Set_*}(K,1),\mathrm{Id})$ in the category $\Set_*$.

\medskip

When $D$ is a left skew brace, the mapping $\Lambda$ is a group homomorphism $$\Lambda\colon (Y,\circ)\to \Aut_\Gp(K,*).$$

\textbf{Notation:} For any two groups $(Y, \otimes)$ and $K,\otimes)$ and any group antihomomorphism $\phi_\otimes\colon (Y, \otimes)\to\Aut_\Gp(K,\otimes)$ , we will denote by $(Y\ltimes_{\phi_\otimes}, \otimes)$ the semidirect product of $Y$ and $K$ via $\phi_\otimes$, where $ \otimes$ is defined as follows $$(y,k) \otimes (y',k'):=(y\otimes y', \phi_{\otimes y'}(k)\otimes k').$$

\begin{theorem}\label{the truth} Let $(D,*,\circ)$ be any digroup, and suppose that $D$ is the semidirect product of its subdigroup $Y$ and its ideal $K$ (i.e., that the equivalent conditions of Proposition~{\rm \ref{0.1}} are satisfied). Let $\varphi_*$, $\varphi_\circ$ and $\Lambda$ be the three mappings defined above. Then there is a digroup isomorphism $$\alpha\colon (Y\times K,+, \circ)\to (D,*,\circ)$$ defined by $\alpha(y,k)= y\circ k$ for every $y\in Y$ and  every $k\in K$, where the operations on the digroup $(Y\times K,+,\circ)$ are defined by \begin{equation}(y,k)+(y',k')=(y* y', (\Lambda_{y* y'})^{-1}(\phi_{* y'}(\Lambda_y(k))* \Lambda_{y'}(k')))
    \label{A}
\end{equation}  and \begin{equation}(y,k) \circ(y',k')=(y\circ y', \phi_{\circ y'}(k)\circ k')\label{B}
\end{equation} for every $(y,k),(y',k')\in Y\times K$. \end{theorem}

\begin{proof} It is well known that $(Y\times K,  \circ)$ is a group. It is straightforward to see that $(Y\times K, +)$ is a monoid. Moreover the inverse of every element $(y,k)\in Y\times K$ is given by $$-(y,k)=(y^{-*}, (\Lambda_{y^{-*}})^{-1}((\phi_{*y^{-*}})(\Lambda_y(k)))^{-*})
,$$ because 
	\begin{equation*}
		\begin{split}
			(y,k)&+(y^{-*}, (\Lambda_{y^{-*}})^{-1}((\phi_{*y})^{-1}(\Lambda_y(k)))^{-*})=\\&= (1,\phi_{* y^{-*}}(\Lambda_y(k))* \Lambda_{y^{-*}}(\Lambda_{y^{-*}})^{-1}((\phi_{*y})^{-1}(\Lambda_y(k)))^{-*})\\&=(1,\phi_{* y^{-*}}(\Lambda_y(k))* ((\phi_{*y})^{-1}(\Lambda_y(k)))^{-*})\\&= (1,\phi_{* y^{-*}}(\Lambda_y(k))* (\phi_{* y^{-*}}(\Lambda_y(k)))^{-*})=(1,1).
		\end{split}
	\end{equation*} 
	Thus $(Y\times K, +,  \circ)$ is a digroup.
	
 Every element of $D$ can be written in a unique way in the form $y\circ k$ with $y\in Y$ and $k\in K$. Hence $\alpha\colon Y\times K\to D$ defined by $\alpha(y,k)= y\circ k$ for every $y\in Y$ and $k\in K$ is a bijection. 

The bijection $\alpha$ is a group isomorphism $(Y\ltimes_{\phi_\circ} K, \circ)\to (D,\circ)$, because 
\begin{equation*}
	\begin{split}
		\alpha(y,k)\circ\alpha(y',k')&= y\circ k\circ y'\circ k'\\&
		=y \circ y'\circ(y')^{-\circ}\circ k\circ y'\circ k'\\&
		=y\circ y'\circ\phi_{\circ y'}(k) \circ k'\\&
		=\alpha (y\circ y',\phi_{\circ y'}(k) \circ k')\\&
		=\alpha((y,k) \circ(y',k')).
	\end{split}
\end{equation*}

It remains to show that the bijection $\alpha$ is also a group homomorphism $$\alpha\colon(Y\times K,+)\to (D,*).$$ To this end, decompose the bijection $\alpha\colon Y\times K\to D$, $\alpha(y,k)= y\circ k=y*\Lambda_y(k)$ as the composite mapping of two bijections: $\beta\colon Y\times K\to Y\times K$ and the group isomorphism $\gamma\colon(Y \ltimes{\phi_*} K,\otimes*)\to (D,*)$. Define the mapping $\beta\colon Y\times K\to Y\times K$ setting $\beta(y, k)=(y,\Lambda_y(k))$ for every $(y,k)\in Y\times K$. Then $\beta$ is a bijection, because the mapping $Y\times K\to Y\times K$, $(y,k)\mapsto (y,(\Lambda_y)^{-1}(k))$ is a two-sided inverse for $\beta$. The mapping $$\gamma\colon(Y\ltimes_{\phi_*} K,\otimes*)\to (D,*)$$ is defined by $\gamma(y,k)=y*k$ and is an isomorphism. The following commutative diagram of sets and mappings represents our setting:
\begin{equation}\xymatrix{(Y\ltimes_{\phi_*} K,\otimes*)\ar[r]^{\gamma} & (D,*)\\
(Y\times K,+)\ar[u]^\beta\ar[ur]_\alpha} \qquad
\xymatrix{(y,\Lambda_y(k))\ar[r]^{\gamma\ \ \ \ \ } & y\circ k=y*\Lambda_y(k)\\
(y,k).\ar[u]^\beta\ar[ur]_\alpha}
    \label{viu}
\end{equation}

It is easy to see that $\alpha=\gamma\beta$. Hence, in order to show that $\alpha$ is a group homomorphism $(Y\times K,+)\to (D,*)$, it suffices to show that the bijection $\beta$ is a group homomorphism $(Y\times K,+)\to (Y\ltimes_{\phi_*} K, *)$. 
Now
\begin{equation*}
	\begin{split}
		\beta((y,k)+(y',k'))&=\beta(y* y', (\Lambda_{y* y'})^{-1}(\phi_{* y'}(\Lambda_y(k))* \Lambda_{y'}(k')))\\&=(y* y', \phi_{* y'}(\Lambda_y(k))* 			\Lambda_{y'}(k'))\\&=(y, \Lambda_y(k)) *(y', \Lambda_{y'}(k'))\\&= \beta(y,k)*\beta(y',k').
	\end{split}
\end{equation*}

This concludes the proof.\end{proof}

Notice that the commutativity of diagram  (\ref{viu}) simply states that $y*\Lambda_y(k)=y\circ k$ for every $y\in Y$, $k\in K$.

\medskip

\begin{remark}{\rm The operations $+$ and $ \circ$ in the statement of Theorem~\ref{the truth} may seem quite complicate, but their naturality appears as soon as we look at how they define the products between an element of the form $(y,1)$ and an element of the form $(k,1)$. In fact, it is very easy to see, from (\ref{A}) and (\ref{B}), that for every $y\in Y$ and every $k\in K$: \begin{equation}\begin{array}{l} (y,1) \circ (1,k)=(y,k) \\ (1,k) \circ(y,1)=(y,\phi_{\circ y}(k)) \\
(y,1)+(1,k)=(y,(\Lambda_y)^{-1}(k)) \\ (1,k)+(y,1)=(y,(\Lambda_y)^{-1}\phi_{*y}(k)).\end{array}\label{ppp}\end{equation} Of these four identities, the first is a trivial identity, the second allows to define the antihomomorphism $\phi_\circ$, the third allows to define the morphism $\Lambda$, and the fourth, in view of the third, allows to define the antihomomorphism $\phi_*$. Clearly, the last three identities in (\ref{ppp}) correspond to the identities  in (\ref{sistema}). Notice that the elements $(y,1)$ and $(1,k)$ generate the digroup $(Y\times K,+,\circ)$, so that the four identities in (\ref{ppp}) completely determine the formulas (\ref{A}) and (\ref{B}) for the operations in $Y\times K$.}\end{remark}

\begin{proposition}\label{vhil} For the digroup $(Y\times K,+, \circ)$, semidirect product of $Y$ and $K$ in the notation above, we have that $$\lambda^{Y\times K}_{(y,k)}(y',k')=(\lambda^Y_y(y'), ((\Lambda_{\lambda^Y_y(y')})^{-1}(\phi_{*(\lambda^Y_y(y'))}(\Lambda_y(k))))^{-*}* \Lambda_{y\circ y'}(\phi_{\circ y'}(k)\circ k')).$$
\end{proposition}

\begin{proof} $$\begin{array}{l}
\lambda^{Y\times K}_{(y,k)}(y',k')=-(y,k)+((y,k) \circ(y',k'))= 
\\ \qquad=\left(y^{-*}, \Lambda_{y^{-*}}^{-1}((\phi_{*y})^{-1}(\Lambda_y(k))^{-*})\right)+(y\circ y', \phi_{\circ y'}(k)\circ k')= 
\\ \qquad=\Bigl(y^{-*}*(y\circ y'), \Lambda_{\lambda^Y_y(y')}^{-1}\bigl(\phi_{*(y\circ y')}(\Lambda_{y^{-*}}(\Lambda_{y^{-*}}^{-1}((\phi_{*y})^{-1}(\Lambda_y(k))^{-*})))\bigr)* \\ \qquad\qquad * \Lambda_{y\circ y'}\bigl(\phi_{\circ y'}(k)\circ k'\bigr)\Bigr)=
\\ \qquad=\Bigl(\lambda^Y_y(y'), \Lambda_{\lambda^Y_y(y')}^{-1}\bigl(\phi_{*(y\circ y')}((\phi_{*y^{-*}})(\Lambda_y(k))^{-*})))\bigr)*\Lambda_{y\circ y'}\bigl(\phi_{\circ y'}(k)\circ k'\bigr)\Bigr)=
\\ \qquad =\Bigl(\lambda^Y_y(y'), \Lambda_{\lambda^Y_y(y')}^{-1}(\phi_{*(\lambda^Y_y(y'))}(\Lambda_y(k))^{-*})* \Lambda_{y\circ y'}(\phi_{\circ y'}(k)\circ k')\Bigr).\end{array}$$
\end{proof}

\begin{remark}\label{bjip} {\rm We are very grateful to Thomas Letourmy, who has found a mistake in a remark in a previous version of this paper. He showed that our remark was wrong, giving the following example. Consider the left skew brace $\Z\times\Z$ with operations defined by 
$$\begin{array}{l}(y,k)*(y',k')=(y+y',k+(-1)^yk') \\ (y,k)\circ (y',k')=(y+y',(-1)^{y'}k+(-1)^yk').\end{array}$$ This left skew brace has an idempotent endomorphism $e\colon (y,k)\mapsto (y,0)$. 
Thus $\Z\times\Z$  is the semidirect product of its subdigroup $Y=\Z\times\{0\}$ and its ideal $K=\ker(e)=\{0\}\times \Z$. Observe that in this case $$\phi_{*y}(k)=(y,0)^{-*}*(0,k)*(y,0)=(-1)^{-y}k,$$ $$ \phi_{\circ y}(k)=(y,0)^{-\circ}\circ(0,k)\circ (y,0)=k$$ and $$\Lambda_{y}(k)=(y,0)^{-*}*((y,0)\circ (0,k))=k.$$
Following Theorem~\ref{the truth}, there exists a skew brace isomorphism $$\alpha\colon (\Z\times \Z, +, \circ)\to (\Z\times \Z, *,\circ)$$ where $$(y,k)+(y',k')=(y+y',(-1)^y k+k')$$ and $$(y,k) \circ(y',k')=(y,k)\oplus (y',k')=(y+y',k+k')$$ The isomorphism $\alpha$ is given by $\alpha(y,k)=(y,(-1)^y k)$ and the commutative diagram (\ref{viu}) becomes in this example the diagram \begin{equation*}\xymatrix{(\Z\ltimes_{\phi_*} \Z, *)\ar[r]^{\gamma} & (\Z\times \Z,*)\\
(\Z\times \Z,+)\ar[u]^\beta\ar[ur]_\alpha} \qquad
\xymatrix{(y,k)\ar[r]^{\gamma\ \ \ \ \ } & (y,(-1)^yk)\\
(y,k).\ar[u]^\beta\ar[ur]_\alpha}
    \end{equation*}}\end{remark}

\section{Outer semidirect product of left skew braces}

Let us construct now the outer semidirect product of two digroups. 

\begin{theorem}\label{thm2} Let $(Y,*,\circ)$ and $(K,*,\circ)$ be digroups and suppose that there are  two group antihomomorphisms $\phi_*\colon (Y, *)\to\Aut_\Gp(K,*)$ and $\phi_\circ\colon (Y, \circ)\to\Aut_\Gp(K,\circ)$, and a morphism $\Lambda\colon (Y,1)\to (S_K,\mathrm{Id})$ in $\Set_*$. On the cartesian product $Y\times K$ define two operations via $$(y,k)+(y',k')=(y* y', (\Lambda_{y* y'})^{-1}(\phi_{* y'}(\Lambda_y(k))* \Lambda_{y'}(k')))$$ 
and $$(y,k) \circ(y',k')=(y\circ y', \phi_{\circ y'}(k)\circ k')$$ for every $y,y'\in Y$ and $k,k'\in K$. Then $(Y\times K,+,\circ)$ is a digroup.\end{theorem}

\begin{proof}It is well known that $(Y\times K, \circ)$ is a group, the semidirect product of the groups $(Y,\circ)$ and $(K,\circ)$. In order to show that the magma $(Y\times K,+)$ is a group, it suffices to show that the mapping $\beta\colon Y\times K\to Y\times K$ defined by $\beta(y, k)=(y,\Lambda_y(k))$ for every $(y,k)\in Y\times K$ is a magma isomorphism between $(Y\times K,+)$ and the group $(Y\times K, *)$, semidirect product of the groups $(Y,*)$ and $(K,*)$ via the group antihomomorphism $\phi_*$.
This is now easy.
\end{proof}

\begin{proposition}\label{hl} Let $(Y,*,\circ)$ and $(K,*,\circ)$ be digroups and suppose that there are  two group antihomomorphisms $\phi_*\colon (Y, *)\to\Aut_\Gp(K,*)$ and $\phi_\circ\colon (Y, \circ)\to\Aut_\Gp(K,\circ)$, and a morphism $\Lambda\colon (Y,1)\to (S_K,\mathrm{Id})$ in $\Set_*$. Then $(Y\times K,+, \circ)$ with the operations defined in Theorem~{\rm \ref{thm2} }is a left skew brace if and only if $Y$ and $K$ are left skew braces, $\Lambda$ is a group homomorphism $\Lambda\colon (Y,\circ)\to\Aut_\Gp(K,*)$, and $$\begin{array}{l}
        \phi_{\circ (y'*y'')}(k)\circ (\Lambda_{y'*y''})^{-1}\bigl(\phi_{*y''}(\Lambda_{y'}(k'))*\Lambda_{y''}(k'')\bigr) = \\ \quad=
        (\Lambda_{y\circ (y'*y'')})^{-1}\bigl(\phi_{*\lambda_{y}^Y(y'')}(\Lambda_{y\circ y'}(\phi_{\circ y'}(k)\circ k')*\Lambda_{y}(k)^{-*})\bigr)*\Lambda_{y\circ y''}(\phi_{\circ y''}(k)\circ k'')
    \end{array}$$
\end{proposition}

\begin{proof}
    By Theorem \ref{thm2} $(Y\times K,+, \circ)$ is a digroup.  If $Y\times K$ is a left skew brace, then so are $Y$ and $K$, and $\Lambda\colon (Y,\circ)\to \Aut_\Gp(K,*)$ is a group homomorphism. Hence assume that these three hypotheses hold ($Y$ and $K$ are left skew braces, and $\Lambda\colon (Y,\circ)\to \Aut_\Gp(K,*)$ is a group homomorphism). Then
$Y\times K$ is a left skew brace if and only if $$(y,k) \circ((y',k')+ (y'',k''))=(y,k) \circ (y',k')-(y,k)+(y,k) \circ (y'',k'').$$

    For the left hand side we have $$\begin{array}{l}            (y,k) \circ\bigl((y',k')  + (y'',k'')\bigr)= \\ \qquad=(y,k) \circ \Bigl(y'* y'',(\Lambda_{y'*y''})^{-1}\bigl(\phi_{*y''}(\Lambda_{y'}(k'))*\Lambda_{y''}(k'')\bigr)\Bigr)=\\ \qquad= 
            \Bigl(y\circ(y'*y''),\phi_{\circ (y'*y'')}(k)\circ (\Lambda_{y'*y''})^{-1}\bigl(\phi_{*y''}(\Lambda_{y'}(k'))*\Lambda_{y''}(k'')\bigr)\Bigr).
        \end{array}$$
    For the right hand side we have 
    \begin{equation*}
        \begin{split}
            (y,&k) \circ(y',k')-(y,k)+(y,k) \circ (y'',k'')=
            \\& \bigl(y\circ y',\phi_{\circ y'}(k)\circ k'\bigr)+\bigl(y^{-*}, (\Lambda_{y^{-*}})^{-1}(\phi_{*y^{-*}}(\Lambda_{y}(k))^{-*})\bigr)+(y\circ y'',\phi_{\circ y''}(k)\circ k'')=
            \\& \Bigl((y\circ y')* y^{-*}, (\Lambda_{(y\circ y')* y^{-*}})^{-1}\bigl(\phi_{* y^{-*}}(\Lambda_{y\circ y'}(\phi_{\circ y'}(k)\circ k'))* \\ & \qquad \qquad \qquad *\Lambda_{y^{-*}}((\Lambda_{y^{-*}})^{-1}(\phi_{*y^{-*}}(\Lambda_{y}(k))^{-*}))\bigr)\Bigr)+ \bigl(y\circ y'',\phi_{\circ y''}(k)\circ k''\bigr)=
            \\& \Bigl((y\circ y')* y^{-*}, (\Lambda_{(y\circ y')* y^{-*}})^{-1}\bigl(\phi_{* y^{-*}}(\Lambda_{y\circ y'}(\phi_{\circ y'}(k)\circ k')*\Lambda_{y}(k)^{-*})\bigr)\Bigr)+ \\ & \qquad \qquad \qquad + \bigl(y\circ y'',\phi_{\circ y''}(k)\circ k''\bigr)=
            \\& \Bigl((y\circ y')* y^{-*}*(y\circ y''),(\Lambda_{(y\circ y')* y^{-*}*(y\circ y'')})^{-1}\bigl(\phi_{*(y\circ y'')}(\Lambda_{(y\circ y')* y^{-*}}((\Lambda_{(y\circ y')* y^{-*}})^{-1}\\&\qquad \qquad\bigl(\phi_{* y^{-*}}(\Lambda_{y\circ y'}(\phi_{\circ y'}(k)\circ k')*\Lambda_{y}(k)^{-*})\bigr))*\Lambda_{y\circ y''}(\phi_{\circ y''}(k)\circ k''))\Bigr)=
            \\& \Bigl((y\circ y')* y^{-*}*(y\circ y''),(\Lambda_{(y\circ y')* y^{-*}*(y\circ y'')})^{-1}\bigl(\phi_{*(y\circ y'')}\bigl(\phi_{* y^{-*}}(\Lambda_{y\circ y'}(\phi_{\circ y'}(k)\circ k')*\\& \qquad \qquad\Lambda_{y}(k)^{-*})\bigr))*\Lambda_{y\circ y''}(\phi_{\circ y''}(k)\circ k'')\Bigr)   
        \end{split}
    \end{equation*}
    So $Y\times K$ is a left skew brace if and only if 
  $$  \begin{array}{l}
        \phi_{\circ (y'*y'')}(k)\circ (\Lambda_{y'*y''})^{-1}\bigl(\phi_{*y''}(\Lambda_{y'}(k'))*\Lambda_{y''}(k'')\bigr) = \\ \quad=
        (\Lambda_{y\circ (y'*y'')})^{-1}\bigl(\phi_{*\lambda_{y}^Y(y'')}(\Lambda_{y\circ y'}(\phi_{\circ y'}(k)\circ k')*\Lambda_{y}(k)^{-*})\bigr))*\Lambda_{y\circ y''}(\phi_{\circ y''}(k)\circ k'')
    \end{array}$$
        \end{proof}

Notice that in the equation in the statement of the previous proposition, the first term does not depend on $y$. Hence, we can replace that equation with two equations as in the statement of the next corollary:

\begin{corollary}
    Let $(Y,*,\circ)$ and $(K,*,\circ)$ be left skew braces. Then $(Y\times K,+, \circ)$ with the operations defined in Theorem \ref{thm2} is a left skew brace if and only if $$\begin{array}{l}
        \phi_{\circ (y'*y'')}(k)\circ (\Lambda_{y'*y''})^{-1}\bigl(\phi_{*y''}(\Lambda_{y'}(k'))*\Lambda_{y''}(k'')\bigr) = \\ \qquad=
        (\Lambda_{y'*y''})^{-1}\bigl(\phi_{*y''}(\Lambda_{y'}(\phi_{\circ y'}(k)\circ k')*k^{-*})\bigr)*\Lambda_{y''}(\phi_{\circ y''}(k)\circ k'')
    \end{array}$$ and $$\begin{array}{l}
        (\Lambda_{y'*y''})^{-1}\bigl(\phi_{*y''}(\Lambda_{y'}(\phi_{\circ y'}(k)\circ k')*k^{-*})\bigr)*\Lambda_{y''}(\phi_{\circ y''}(k)\circ k'') = \\ \ =
        (\Lambda_{y\circ (y'*y'')})^{-1}\bigl(\phi_{*\lambda_{y}^Y(y'')}(\Lambda_{y\circ y'}(\phi_{\circ y'}(k)\circ k')*\Lambda_{y}(k)^{-*})\bigr)*\Lambda_{y\circ y''}(\phi_{\circ y''}(k)\circ k'').
    \end{array}$$
    \end{corollary}
    
    \begin{remark}
	Recall that a left skew brace $(Y,*,\circ)$ is abelian in the sense of \cite{BFP} if the operations $*$ and $\circ$ coincide  and they are commutative. So suppose to have two abelian left skew braces $Y$ and $K$. In the notation of Proposition \ref{hl} if $\Lambda_y=\mathrm{Id}_K$ for every $y\in Y$ and $\phi_*=\phi_\circ=\phi$, then we obtain the usual semidirect product of groups.
\end{remark}

\begin{remark} One might wrongly think that the three mappings $\phi_*$, $\phi_\circ$ and $\Lambda$ are not completely independent. In order to see that this is not the case, let
us determine all possible semidirect products of the left skew brace $(\Z,+,+) $ and itself. (The example in Remark~\ref{bjip} is such an example). Since $\Aut_\Gp(\Z,+)\cong\Z/2\Z$, there are only two group morphisms $(\Z,+)\to \Aut_\Gp(\Z,+)$: one is the trivial morphism $\iota\colon (\Z,+)\to \Aut_\Gp(\Z,+)$ defined by $\iota(x)(y)=y$ for every $x,y\in \Z$, and the other is the group morpism $\alpha\colon (\Z,+)\to \Aut_\Gp(\Z,+)$ defined by $\alpha(x)(y)=(-1)^xy$ for every $x,y\in \Z$. Hence we have $2^3=8$ possible cases for the three morphisms $\phi_*,\phi_\circ$ and $\Lambda$, corresponding to the eight rows in the following table:
\begin{center}
  \begin{tabular}{ p{0.8cm} | c | c | c | c | c | p{1.8cm} }
The eight cases & $\phi_*$ & $\phi_\circ$ & $\Lambda$ & $(y,k)+(y',k')=$  & $(y,k)\circ(y',k')=$ & Is $Y\times K$ a left skew brace? \\
\hline
(1) & $\iota$ & $\iota$ & $\iota$ & $(y+y',k+k')$  & $(y+y',k+k')$ & Yes \\
\hline
(2) & $\iota$ & $\iota$ & $\alpha$ & $(y+y',(-1)^{y'}k+(-1)^{y}k')$  & $(y+y',k+k')$ & No \\
\hline
(3) & $\iota$ & $\alpha$ & $\iota$ & $(y+y',k+k')$  & $(y+y',(-1)^{y'}k+k')$ & No \\
\hline
(4) & $\iota$ & $\alpha$ & $\alpha$ & $(y+y',(-1)^{y'}k+(-1)^{y}k')$  & $(y+y',(-1)^{y'}k+k')$ & No \\
\hline
(5) & $\alpha$ & $\iota$ & $\iota$ & $(y+y',(-1)^{y'}k+k')$  & $(y+y',k+k')$ & Yes \\
\hline
(6) & $\alpha$ & $\iota$ & $\alpha$ & $(y+y',k+(-1)^{y}k')$  & $(y+y',k+k')$ & No \\
\hline
(7) & $\alpha$ & $\alpha$ & $\iota$ & $(y+y',(-1)^{y'}k+k')$  & $(y+y',k+(-1)^{y}k')$ & Yes \\
\hline
(8) & $\alpha$ & $\alpha$ & $\alpha$ & $(y+y',k+(-1)^{y}k')$  & $(y+y',k+(-1)^{y}k')$ & Yes \\
\end{tabular}\end{center}

\bigskip

For instance: 

(a) In all the Examples (1)-(8), $\Lambda_y$ is always $\iota$ or $\alpha$. In both cases $(\Lambda_y)^{-1}=\Lambda_y$.

(b) Example (1) is the left skew brace $(\Z\oplus\Z, +,+)$, where $(\Z\oplus\Z, +)$ is the direct sum of two copies of the group $\Z$. Hence it is a left skew brace.

(c) In Example (2), the formula in the statement of Proposition \ref{hl} becomes $$(-1)^{y'+y''}k+k'+k''=(-1)^{y'}k+k'-k+{y''}k+k'',$$ which is false for $y'=y''=k=1$ and $k'=k''=0$. Therefore the digroup of Example (2) is not a left skew brace.

(d) The example in Remark~\ref{bjip} corresponds to the case (5). 
\end{remark}

\section{Actions and semidirect products of digroups}

What we have seen in the previous paragraph allows us to define actions of digroups on other digroups.

An {\em action} of a digroup $D$ on a digroup $K$ is a triplet $(\phi_*,\phi_{\circ},\Lambda)$, where $$\phi_*\colon(D,*)\to \Aut(K,*)\quad\mbox{\rm and}\quad\phi_{\circ}\colon (D,\circ)\to \Aut(K,\circ)$$ are group antihomomorphisms, and $\Lambda\colon (D,1)\to (\Aut_{\Set_*}(K,1),\mathrm{Id})$ is a morphism in the category $\Set_*$ of pointed sets. More formally:

\begin{definition}{\rm Let $(Y,*,\circ)$ be a digroup.
A {\em $Y$-digroup} $(K,*,\circ,\phi_{*}, \phi_{\circ},\Lambda)$ is a digroup $(K,*,\circ)$ with two group antihomomorphisms $\phi_{*}\colon (Y,*)\to\Aut(K,*)$ and $ \phi_{\circ}\colon (Y,\circ)\to\Aut(K,\circ)$ and a morphism $\Lambda\colon  (D,1)\to (\Aut_{\Set_*}(K,1),\mathrm{Id})$ in the category $\Set_*$. }\end{definition}

The most natural example of a $Y$-digroup is the $Y$-digroup $(Y,*,\circ,\phi_{*},\phi_{\circ},\lambda^Y)$, where \begin{equation}\label{alpha}\phi_{*}\colon (Y,*)\to\Aut(Y,*)\quad\mbox{\rm  and}\quad \phi_{\circ}\colon (Y,\circ)\to\Aut(Y,\circ) \end{equation} are defined by $$\phi_{*}(x)(y)=x^{-*}*y*x\quad\mbox{\rm  and}\quad \phi_{\circ}(x)(y)=x^{-\circ}\circ y\circ x,$$ respectively, and $\Lambda\colon (Y,1)\to (\Aut_{\Set_*}(Y,1),\mathrm{Id})$ is defined by $\lambda^Y(x)(y)=x^{-*}*(x\circ y)$., for every $x,y\in Y$.

\bigskip

According to Theorem~\ref{the truth}, given a $(Y,*,\circ)$-digroup $(K,*,\circ,\phi_*,\phi_\circ,\Lambda)$, on the cartesian product $Y\times K$ we can define two operations via $$(y,k)+(y',k')=(y* y', (\Lambda_{y* y'})^{-1}(\phi_{* y'}(\Lambda_y(k))* \Lambda_{y'}(k')))$$ and $$(y,k) \circ(y',k')=(y\circ y', \phi_{\circ y'}(k)\circ k')$$ for every $(y,k),(y',k')\in Y\times K$, getting a digroup.

\begin{definition}{\rm Let $(Y,*,\circ)$ be a left skew braces.
A {\em $Y$-skew brace} $(K,*,\circ,\phi_{*}, \phi_{\circ},\Lambda)$ is a left skew brace $(K,*,\circ)$ with two group antihomomorphisms $\phi_{*}\colon (Y,*)\to\Aut(K,*)$ and $ \phi_{\circ}\colon (Y,\circ)\to\Aut(K,\circ)$ and a group homomorphism $\Lambda\colon  (Y,\circ)\to \Aut(K,*)$ such that 
$$\begin{array}{l}
        \phi_{\circ (y'*y'')}(k)\circ (\Lambda_{y'*y''})^{-1}\bigl(\phi_{*y''}(\Lambda_{y'}(k'))*\Lambda_{y''}(k'')\bigr) = \\ \quad=
        (\Lambda_{y\circ (y'*y'')})^{-1}\bigl(\phi_{*\lambda_{y}^Y(y'')}(\Lambda_{y\circ y'}(\phi_{\circ y'}(k)\circ k')*\Lambda_{y}(k)^{-*})\bigr)*\Lambda_{y\circ y''}(\phi_{\circ y''}(k)\circ k'')
    \end{array}$$
 }\end{definition}

In this case, we can construct the outer semidirect product of the two left skew braces, which turns out to be a left skew brace.

\bigskip

Let us give a further example of a left skew brace $(A,*,\circ)$ that is a semidirect product of left skew braces \cite[Example 1.13]{SV}.

\begin{example} {\rm In this example the group $(D,*)$ is the symmetric group $S_3$, and the group $(D,\circ)$ is the cyclic group $C_6$ of order 6 generated by the transposition $(1\ 2)$. The group morphism $\lambda\colon C_6\to\Aut(S_3)$ maps the identity, and the two cycles of order 3 to the identity of $\Aut(S_3)$, and maps the three transpositions to the conjugation by $(2\ 3)$. Here the powers $(1\ 2)$, $(1\ 2)^2$, $(1\ 2)^3$, $(1\ 2)^4$, $(1\ 2)^5$, $(1\ 2)^6$ of $(1\ 2)$ in $(D,\circ)$  are $(1\ 2)$, $(1\ 3\ 2)$, $(2\ 3)$, $(1\ 2\ 3)$, $(1\ 3), \id$, respectively. If we consider the mapping $e\colon (D,*)\to (D,*)$ that maps any element $f$ of $S_3$ to $(2\ 3)^{\sgn(f)}$, we clearly get an idempotent endomorphism of the group $(D,*)$. But the same mapping $e$ is an endomorphism of $(D,\circ)$ as well, because it is the mapping that sends the even powers of $(1\ 2)$ in $(D,\circ)$ to $\id$, and the odd  powers of $(1\ 2)$ in $(D,\circ)$ to the element $(2\ 3)$ of order 2 of $(D,\circ)$. Hence the left skew brace $(D,*,\circ)$ is a semidirect product of the ideal $A_3$ and the skew subbrace $\{\id, (2\ 3)\}$.

In this example $K$ is the cyclic group with three elements and $Y$ is the cyclic group with two elements. (Both are abelian left skew braces.) Hence $\Aut(K)$ has exactly two elements (the identity $\iota$ and an involution), and so there are exactly two group (anti)homomorphisms $Y\to\Aut(K)$ (the trivial homomorphism $i$  that sends both elements of $Y$ to $\iota$ and the unique group isomorphism $f\colon Y\to\Aut(K)$). 
For their semidirect product extension $D$, we find that $(D,*)$ is the symmetric group $S_3$, and $(D,\circ)$ is the cyclic group with six elements. Since $(D,*)\cong S_3$ is not abelian, we have that $\phi_*=f$. Since $(D,\circ)$ is the cyclic group with six elements, hence it is abelian, we get that $\phi_\circ=i$. It is now easy to show that $\Lambda=f$, as well.}\end{example}

\section{Relation with centers and commutators}\label{5}

The commutator of two ideals $I$ and $J$ of a left skew brace $A$ is the ideal of $A$ generated by the commutator $[I,J]_{(A,*)}$ in the group $(A,*)$, the commutator $[I,J]_{(A,\circ)}$ in the group $(A,\circ)$, and all the elements $(i\circ j)^{-*}*i*j$ with $i\in I$ and $j\in J$ \cite{BFP}.

\begin{Lemma}\label{comm} For every pair of ideals $I$ and $J$ of a left skew brace $A$, one has $[I,J]=[J,I]$.\end{Lemma}

\begin{proof} For groups, one has that $[I,J]_{(A,*)}=[J,I]_{(A,*)}$ and $[I,J]_{(A,\circ)}=[J,I]_{(A,\circ)}$. Hence it remains to show that if $(A,*,\circ)$ is a left skew brace with the groups $(A,*)$ and $(A,\circ)$ abelian and $(i\circ j)^{-*}*i*j=1$ for every $i\in I$ and $j\in J$, then  $(j\circ i)^{-*}*j*i=1$ for every $i\in I$ and $j\in J$. Now $(i\circ j)^{-*}*i*j=1$ implies that $i\circ j=i*j$, so that $j\circ i=i\circ j=i*j=j*i$. Hence $(j\circ i)^{-*}*j*i=1$.\end{proof}

Recall that in the lattice of all the ideals of a left skew brace $A$, one has that $I\vee J=I*J=I\circ J$ and $I\wedge J=I\cap J$. This is a multiplicative lattice \cite{FFJ} in which multiplication is commutative in view of Lemma~\ref{comm}.

\begin{Lemma} If $I,J,K$ are ideals of a left skew brace $A$, then $[I,J*K]=[I,J]*[I,K]$.\end{Lemma}

\begin{proof} From $J*K\supseteq J,K$, it follows that $[I,J*K]\supseteq [I,J],[I,K]$. Therefore $[I,J*K]\supseteq [I,J]*[I,K]= [I,J]\circ [I,K]$. Now it is known that the equality $[I,JK]=[I,J][I,K]$ holds for normal subgroups $I,J,K$ of a group $G$. Hence, to conclude the proof, it suffices to prove that if $(A,*,\circ)$ is a brace in which both the groups $(A,*)$ and $(A,\circ)$ are abelian, and $i*j=i\circ j$, $i*k=i\circ k$ for every $i\in I$, $j\in J$ and $k\in K$, then $i*(j\circ k)=i\circ(j\circ k)$ for every $i,j,k$. Now $i*j=i\circ j$ for every $i$ and $j$ can be restated saying that $J$ is contained in the kernel of the group morphism $\lambda|^I\colon (A,\circ)\to\Aut(I,*)$. Similarly, $K$ is contained in the kernel of the group morphism $\lambda|^I$. Therefore $J\circ K$ is contained in that kernel, that is $i*x=i\circ x$ for every $x\in J\circ K=J*K$, as desired.
\end{proof}

Recall that it is possible to define the {\em right center} $rZ(x):=x\wedge \rann_L(x)$ and the
{\em left center} $lZ(x):=x\wedge \lann_L(x)$ of any element $x$ of a multiplicative lattice $L$. The  {\em right center} $rZ(L)$ and   the
{\em left center} $lZ(L)$ of a multiplicative lattice $L$ are the right center and the left center of the greatest element $1$ of $L$. In our special case of left skew braces, the commutator $[I,J]$ of two ideals $I,J$ is equal to $[J,I]$, so that right center and left center of a left skew brace $A$ coincide. It is the greatest ideal $Z$ of $A$ such that $[Z,A]=1$. Hence $Z=\{\, z\in A\mid a*z=z*a, a\circ z=z\circ a$ and $a*z=a\circ z$ for every $a\in A\,\}$.

%
%
%

\end{document}